\documentclass{article}

\usepackage[utf8]{inputenc}
\usepackage[T1]{fontenc}
\usepackage[english]{babel}
\usepackage{amsmath}
\usepackage{amsthm}
\usepackage{amssymb}
\usepackage{amsfonts}
\usepackage{enumitem}
\usepackage{multirow}
\usepackage{dsfont}
\usepackage{relsize}
\usepackage{graphicx}
\usepackage[lofdepth,lotdepth]{subfig}

\newtheorem{theorem}{Theorem}[section]
\newtheorem{proposition}[theorem]{Proposition}
\newtheorem{lemma}[theorem]{Lemma}
\newtheorem{conjecture}[theorem]{Conjecture}

\theoremstyle{remark}
\newtheorem{rmk}[theorem]{Remark}
\newtheorem{terminology}[theorem]{Terminology}
\newtheorem*{notation}{Notation}
\numberwithin{equation}{section}

\newcommand{\R}{\mathbb{R}}
\newcommand{\N}{\mathbb{N}}
\newcommand{\Z}{\mathbb{Z}}
\newcommand{\rn}{\mathbb{R}^n}
\newcommand{\hau}{\mathrm{d}\mathcal{H}^{n-1}}

\newcommand{\s}{\partial B}

\newcommand\eaba{\mathcal{E}_{\alpha,\beta,A}}
\newcommand\eg{\mathcal{E}_{\alpha,\beta,A,\gamma}}

\newcommand\V[1]{\mathcal{V}_\alpha(#1)}
\newcommand\U[1]{\mathcal{U}_{\beta,A}(#1)}
\newcommand\uba{\mathcal{U}_{\beta,A}}

\newcommand\abs[1]{\left|#1\right|}
\newcommand\nldc[1]{\|#1\|_{L^2(\s)}^2} 
\newcommand\ncu[1]{\|#1\|_{C^1(\s)}}
\newcommand{\tend}[2]{\displaystyle\mathop{\longrightarrow}_{#1\rightarrow#2}}

\newcommand\1{\mathds{1}}
\newcommand\vat{\mathcal{V}_{\alpha,T}}

\title{Large volume minimizers of a non-local isoperimetric problem: theoretical and numerical approaches}
\author{François Générau, Edouard Oudet}

\begin{document}
\maketitle
\begin{abstract}
  We consider the volume-constrained minimization of the sum of the perimeter and the Riesz potential. We add an external potential of the form $\abs{x}^\beta$ that provides the existence of a minimizer for any volume constraint, and we study the geometry of large volume minimizers. Then we provide a numerical method to address this variational problem.
\end{abstract}
\section{Introduction}
Gamow's liquid drop model for the atomic nucleus consists in:
\begin{equation}\label{eq:liquid drop}
  \inf_{E\subset \rn, \abs{E}=m}  P(E)+\V{E}.
\end{equation}
where
\begin{itemize}
  \item $P(E)$ is the perimeter of $E$,
  \item $\V{E}:=\int_{E \times E }\frac{\mathrm{d}x\mathrm{d}y}{\abs{x-y}^{n-\alpha}}$,
  \item $n \in \N^*$ (the dimension of the ambient space), $m>0$ (called the \emph{mass}) and $\alpha\in(0,n)$ are constants.
\end{itemize}
More precisely, the physical case corresponds to $n=3$ and $\alpha=2$. As shown in \cite{diblock}, this model is also related to diblock copolymers. Problem \eqref{eq:liquid drop} has been studied as an interesting extension of the classical isoperimetric problem. Indeed, two terms are competing: the perimeter tends to round things up (and is minimized by balls), whereas the non-local $\mathcal{V}_{\alpha}$ term, which can be viewed as an electrostatic energy if $n=3$ and $\alpha=2$, tends to spread the mass (and is maximized by balls).
It was shown in \cite{vincent} that if the mass $m$ is small enough, then the problem
\eqref{eq:liquid drop} admits a unique minimizer (up to translation), namely the ball of volume $m$
(see also \cite{knupfer_planar}, \cite{knupfer_general} and \cite{bonacini_cristoferi} for partial results).
On the other hand, for $\alpha \in (0,2)$, it  was shown in \cite{knupfer_general} that for $m$ large enough there is no minimizer of problem \eqref{eq:liquid drop}. This result was simultaneously proved in \cite{nonexistence_lu_otto} in the physical case.
See also \cite{nonexistence3d} for a short proof with a quantitative bound.

To restore the existence of a minimizer for large masses, we add the energy associated to the potential $A\abs{x}^\beta$ to our functional, as we expect it to counter the spreading effect of the $\mathcal{V}_\alpha$ term. Thus we are interested in the following modification of the original problem \eqref{eq:liquid drop}:
\begin{equation}\label{eq:isoperimetric problem}
  \inf_{E\subset \rn, \abs{E}=m} \eaba(E) := P(E)+\V{E}+\uba(E),
\end{equation}
where
\begin{itemize}
  \item $\uba:=\int_E A\abs{x}^\beta \mathrm{d}x$,
  \item $A\geq 0$ is constant.
\end{itemize}

See also \cite{alama2017droplet} and \cite{alama2017ground} (which appeared independently and simultaneously with this work), where the authors use a different and interesting confining potential.

As easily proved in section \ref{sec:existence}, we indeed have the existence of a minimizer in \eqref{eq:isoperimetric problem} for any mass $m$. In section~\ref{sec:extension}, we extend some known results about minimality of small balls, and the domain (of masses $m$) of local minimality of balls. We don't give complete proofs, but recall briefly the techniques used in \cite{vincent} to get these results.

In section~\ref{sec:large volume minimizers}, we study large volume minimizers (\textit{i.e.} when $m$ is large) of \eqref{eq:isoperimetric problem} when $\alpha<\beta$. More precisely we prove the following theorems:
\begin{theorem}\label{thm:hausdorff boundary}
  Given $\alpha \in (0,n)$, $\beta>0$ and $A>0$, assume $\alpha<\beta$. Let
  $(E_m)_{m>0}$ be a family of minimizers in \eqref{eq:isoperimetric problem},
  such that $\abs{E_m}=m$, and let $E_m^*$ be the rescaling of
  $E_m$ of the same mass as the unit ball $B$ (ie $E_m^*=\left(\frac{ \abs{B}}{m}\right)^{ \frac{1}{n}}E_m$).
  Then the boundaries of the sets $(E_m^*)$ Hausdorff-converge to the boundary of $B$ as ${m\to +\infty}$.
\end{theorem}
From proposition \ref{prop:stability}, we know that if $\beta<1$ then large volume minimizers are not exactly balls, but if we assume in addition that $\alpha>1$, then we have:
\begin{theorem}\label{thm:ball unique minimizer}
  Given $\alpha \in (1,n)$, $\beta>0$ and $A>0$, assume $\alpha<\beta$. There exists a mass $m_1=m_1(n,\alpha,\beta,A)>0$ such that for any $m>m_1$ the ball of volume $m$ centered at $0$ is the unique minimizer \eqref{eq:isoperimetric problem}.
\end{theorem}

In section \ref{sec:numerical}, we present a numerical method for problem \eqref{eq:isoperimetric problem}. We also apply this method to the original problem \eqref{eq:liquid drop}. Indeed, the theoretical knowledge we have so far on problem \eqref{eq:liquid drop} raises to natural questions. Is it always the case (\text{i.e.} for any value of $\alpha \in (0,n)$) that there is no minimizer for $m$ big enough?
Is there a set of parameters $n$, $\alpha$ and $m$, such that there exists a minimizer that is different from a ball? Our numerical results indicate that in dimension $2$, the answers are positive and negative respectively.
\section{Existence of a minimizer in \eqref{eq:isoperimetric problem}}\label{sec:existence}
In this section, we prove the following easy proposition:

\begin{proposition}\label{prop:existence}
	As long as $A>0$, problem \eqref{eq:isoperimetric problem} admits a minimizer for any mass $m>0$.
\end{proposition}
\begin{notation}
  We denote by $B$ the unit ball of $\rn$, and by $B[m]$ the ball of volume $m$ centered at $0$.
\end{notation}
\begin{proof}
  Let $(E_k)$ be a minimizing sequence for the variational problem \eqref{eq:isoperimetric problem}. By replacing $E_k$ with the ball $B[m]$ if necessary, we can assume
  \begin{equation}\label{eq:min}
    \eaba(E_k) \leq \eaba(B[m]), \quad \text{for all } k\in \N.
  \end{equation}
  Set $g(x)=A\abs{x}^{\beta}$. As $g(x)\tend{\abs{x}}{+\infty}+\infty$, we can take a sequence of positive radius $(R_k)_{k\in \N}$ and a sequence of positive constants $(A_j)_{j\in\N}$ such that $A_j \tend{j}{\infty}+\infty$ and for all $x \notin B_{R_j},\; g(x)>A_j$. For any $j\in\N$, the sequence $(E_k \cap B_{R_j})_{k\in\N}$ is a sequence of uniformly bounded borel sets, with uniformly bounded perimeters. Indeed, the inequalities $P(E_k \cap B_{R_j}) \leq P(E_k)$, $P \leq \eaba$ and \eqref{eq:min} together give $P(E_k \cap B_{R_j}) \leq \eaba(B[m])$ for all $k$.

Therefore we can extract a $L^1$-converging subsequence of $(E_k \cap B_{R_j})_{k\in \N}$. Doing that for all $j\in\N$ and using a diagonal argument, we get a subsequence of $(E_k)_{k\in \N}$ that converges locally in $L^1$ to a borel set $E\subset \rn$. Using the lower semi-continuity of the perimeter and Fatou's lemma in $\V{E_k}$ and $\uba(E_k)$, we get that
\begin{equation}\label{eq:semicontinuity}
  \eaba(E) \leq \liminf\limits_{k \to \infty}\eaba(E_k).
\end{equation}

Now we show that $\abs{E}=m$. By Fatou's lemma, from $\abs{E_k}=m$, we get $\abs{E} \leq m$. Also, for any $j\in\N$, from the inequalities $\uba(E_k\setminus B_{R_j}) \leq \uba(E_k)$, $\uba \leq \eaba$ and \eqref{eq:min} we get
\begin{equation}\label{eq:estimate1}
\uba(E_k\setminus B_{R_j}) \leq \eaba(B[m]).
\end{equation}
But
\begin{equation}\label{eq:estimate2}
\uba(E_k\setminus B_{R_j})=\int_{E_k\setminus B_{R_j}} g(x) \mathrm{d} x \geq A_j \abs{E_k\setminus B_{R_j}}.
\end{equation}
Thus \eqref{eq:estimate1} and \eqref{eq:estimate2} together give
\[ \abs{E_k\cap B_{R_j}}=m-\abs{E_k\setminus B_{R_j}} \geq m-\frac{\eaba{B[m]}}{A_j}.
\]
Taking the limit $k \to \infty$, then $j \to \infty$, we obtain
\[\abs{E\cap B_{R_j}} \geq m-\frac{\eaba{B[m]}}{A_j}, \quad \text{then} \quad  \abs{E} \geq m.
\]
Thus $\abs{E}=m$. With \eqref{eq:semicontinuity}, it means that $E$ is a minimizer of the variational problem \eqref{eq:isoperimetric problem}.
\end{proof}

\begin{rmk}
It is clear from the proof that proposition \ref{prop:existence} is true if we replace the potential $A\abs{x}^{\beta}$ by any $L^{1}_{loc}$ non-negative function $g$ such that $g(x)\tend{\abs{x}}{+\infty}+\infty$.
\end{rmk}

\section{Extension of some known results}\label{sec:extension}
In this section we recall two known results about the variational problem \eqref{eq:liquid drop}, and extend it to \eqref{eq:isoperimetric problem}, recalling only the techniques used in \cite{vincent}. The first result state that if the mass $m$ is small enough, then problem \eqref{eq:liquid drop} admits a unique (up to translation) minimizer, namely the ball of volume $m$. The same holds for problem \eqref{eq:isoperimetric problem}:
\begin{proposition}\label{thm:small mass}
Given $\alpha\in (0,n)$, $\beta>0$, $A>0$, there exists a constant $m_0(n,\alpha,\beta,A)>0$ such that for any $m\in (0,m_0)$, problem \eqref{eq:isoperimetric problem} admits the ball of volume $m$ centered at $0$ as its unique minimizer.
\end{proposition}
It is a direct consequence of the same theorem for problem \eqref{eq:liquid drop} (see \cite[theorem 1.3]{vincent}), as balls centered at $0$ are also volume-constrained minimizers of $\uba$. This last fact is a consequence of Riesz inequality regarding symmetric decreasing rearrangements (see \cite{analysis_lieb} for rearrangement inequalities). Note that proposition \ref{thm:small mass} is true if we replace the potential $A\abs{x}^{\beta}$ with a symmetric non-decreasing function $g$.

The second result deals with local minimality of balls.
\begin{terminology}
  We say that $E\subset \rn$ is a $L^1$-local minimizer in \eqref{eq:isoperimetric problem} if there exists $\epsilon>0$ such that for any set $F\subset \rn$ such that $\abs{F}=\abs{E}$ and $\abs{E \Delta F}<\epsilon$, $\eaba(E)\leq \eaba(F)$.
\end{terminology}
In the case of problem $\eqref{eq:liquid drop}$, we know from \cite[theorem 1.5]{vincent} that there exists a $m_*>0$ such that if $m<m_*$, then $B[m]$ is a $L^1$-local minimizer in \eqref{eq:liquid drop}, and if $m>m_*$ then $B[m]$ is \emph{not} a $L^1$-local minimizer in \eqref{eq:liquid drop}. As stated in the next theorem, the addition of the $\uba$ term may modify this situation, but we can still apply the techniques used in \cite{vincent} to get a similar result.

\begin{proposition}\label{prop:stability}
Given $\alpha\in (0,n)$, $\beta>0$ and $A>0$,
\begin{enumerate}[label=(\roman*)]
\item\label{item:alpha>beta}
if $\alpha>\beta$, then there exists a mass $m_*(n,\alpha,\beta,A)>0$ such that if $m<m_*$, then $B[m]$ is a $L^1$-local minimizer in \eqref{eq:isoperimetric problem}, and if $m>m_*$ then $B[m]$ is \emph{not} a $L^1$-local minimizer in \eqref{eq:isoperimetric problem},
\item\label{item:alpha=beta}
if $\alpha=\beta$, then either the same holds, or (if $\alpha>1$ and $A$ is small enough) $B[m]$ is a $L^1$-local minimizer in \eqref{eq:isoperimetric problem} for any $m>0$,
\item\label{item:alpha<beta>1}
if $\alpha<\beta$ and $\beta>1$, then there exists a mass $m_*(n,\alpha,\beta,A)>0$ such that if $m>m_*$ then $B[m]$ is a $L^1$-local minimizer in \eqref{eq:isoperimetric problem},
\item\label{item:alpha<beta<1}
if $\alpha<\beta$ and $\beta<1$, then there exists a mass $m_*(n,\alpha,\beta,A)>0$ such that if $m>m_*$ then $B[m]$ is \emph{not} a $L^1$-local minimizer in \eqref{eq:isoperimetric problem},
\item\label{item:alpha<beta=1}
if $\alpha<\beta$ and $\beta=1$, then the conclusion of either \ref{item:alpha<beta>1} or \ref{item:alpha<beta<1} holds (depending on the value of $A$).
\end{enumerate}
\end{proposition}

\begin{rmk}
  The conclusions in points \ref{item:alpha<beta>1}, \ref{item:alpha<beta<1} and \ref{item:alpha<beta=1} are less precise than in points \ref{item:alpha>beta} and \ref{item:alpha=beta}.
\end{rmk}

\begin{proof}[Ideas of the proof.]
  The method used in \cite{vincent} still applies to our functional $\eaba=P+\mathcal{V}_{\alpha}+\uba$. Given $m>0$, let us procede to a rescaling of the functional and set
  \[\gamma=\left(\frac{m}{\abs{B}}\right)^{1/n}\quad \text{and}\quad\eg:=P+\gamma^{1+\alpha}\mathcal{V}_\alpha+\gamma^{1+\beta}\uba,\]
  so that for any set $E$ of volume $m$, the set $E^*=\frac{1}{\gamma}E$ has volume $\abs{B}$ and
  \[\eaba(E)=\gamma^{n-1}\eg(E^*).\]
  Thus we are reduced to finding the $\gamma>0$ such that the unit ball $B$ is a local minimizer of $\eg$.

  Following \cite[section 6]{vincent} we can compute the second variation of $\eg$ at $B$. The terms $P$ and $\mathcal{V}_{\alpha}$ are treated in \cite{vincent} and the term $\uba$ adds no further difficulty. We find that given any smooth compactly supported vector field $X$, such that the volume of $B$ is preserved under the flow $(\Phi_t^X)_{t>0}$ of X, we have:
  \begin{equation}
    \delta^2\eg(B)[X]=\sum_{k \geq 2}(\lambda_k-\lambda_1)\left(1-\gamma^{1+\alpha}\frac{\mu_k^\alpha-\mu_1^\alpha}{\lambda_k-\lambda_1} +\gamma^{1+\beta} \frac{A \beta}{\lambda_k-\lambda_1}\right)a_k(X\cdot \nu_B)^2,
    \nonumber
  \end{equation}
  where
  \begin{itemize}
    \item $\delta^2\eg(B)[X]:=\frac{\mathrm{d^2}}{\mathrm{d}t^2}[\eg(\Phi_t^X(B))]_{t=0}$,
    \item $\nu_B$ is the unit outer normal vector to $\s$,
    \item $a_k(X\cdot \nu_B)$ are the coefficient of the function $X\cdot \nu_B : \s \to \R$ with respect to an othonormal basis of spherical harmonics,
    \item $\lambda_k=k(n+k-2)$ is the $k$-th eigen value of the laplacian on the sphere~$\s$,
    \item $\mu_k^\alpha$ is the $k$-th eigen value of the operator $\mathcal{R}_\alpha$
    defined by \[\mathcal{R}_\alpha u(x):=2\int_{\s}\frac{u(x)-u(y)}{\abs{x-y}^{n-\alpha}}\hau(y),\quad \forall u\in C^1(\s).\]
  \end{itemize}
  From there we deduce that, defining
  \begin{equation}\label{eq:stability set}
    S_*=\{\gamma>0: 1-\gamma^{1+\alpha}\frac{\mu_k^\alpha-\mu_1^\alpha}{\lambda_k-\lambda_1} +\gamma^{1+\beta} \frac{A \beta}{\lambda_k-\lambda_1} \geq 0, \forall k\geq 2 \},
  \end{equation}
  if $\gamma\notin S_*$, then there exists a vector field $X$ such that
  \[\abs{\phi_t^X(B)}=\abs{B}\quad \text{and} \quad \eg(\phi_t^X(B))<\eg(B) \quad \text{for $t$ small enough.}\]
  Thus $B$ is not a $L^1$-local minimizer of $\eg$ if $\gamma\notin S_*$.

  Now let us set
  \begin{equation}\label{eq:local min set}
    \widetilde{S}_*=\{\gamma>0: 1-\gamma^{1+\alpha}\frac{\mu_k^\alpha-\mu_1^\alpha}{\lambda_k-\lambda_1} +\gamma^{1+\beta} \frac{A \beta}{\lambda_k-\lambda_1} > 0, \forall k\geq 2 \}.
  \end{equation}
  We assume $\gamma\in \widetilde{S}_*$ and explain how to show that $B$ is a $L^1$-local minimizer of $\eg$. First, we note that it is true in a certain class of nearly spherical sets. More precisely, let $E$ be a nearly spherical set associated to a $C^1$ function $u:\s \to \R$:
  \[E :=\{s(1+u(x))x, ~ x\in \s,~ 0 \leq s \leq 1 \}.\]
  Assume that $\abs {E}=\abs{B}$ and $\int_{E}\abs{x}^{\beta -2}x\mathrm{d}x=0$. Then using some explicit computations and taylor expansions, we can show that there exist some constants $\epsilon(n,\alpha,\beta,A,\gamma)>0$ and  $C(n,\alpha,\beta,A,\gamma)>0$ such that if $\ncu{u} \leq \epsilon(n,\alpha,\beta,A,\gamma) $, then
  \begin{equation}\label{eq:almost spherical}
    \eg(E)-\eg(B) \geq C(n,\alpha,\beta,A,\gamma)\left( \nldc{u}+\nldc{\nabla (u)} \right).
  \end{equation}
  Next we argue by contradiction and assume that we have a sequence of borel sets $(E_k)$ such that for any $k$, $\abs{E_k}=\abs{B}$, $\eg(E_k) <\eg(B[m])$ and $\abs{E_k \Delta B} \tend{k}{\infty} 0$. We consider a set $F_k$ solution of the penalized problem:
  \begin{equation}\nonumber
  \inf \{\eg(E)+M \abs{E \Delta E_k}, E \subset \rn\},
  \end{equation}
  with $M>0$ to be taken large enough. The role of the set $F_k$ is to be "close to $E_k$", and to be a $\Lambda$-minimizer in the sense that
  \[P(F_k)\leq P(E)+\Lambda \abs{E \Delta F_k}, \quad \text{for any borel set $E$.}\]
  Thus we show that $F_k$ is a $\Lambda$-minmizer for some $\Lambda$ uniform in $k$, and that $\abs{F_k \Delta B}\tend{k}{\infty}0$, which implies by classical regularity theory that $F_k$ is an almost spherical set. Up to translating and rescaling $F_k$ we can apply inequality \eqref{eq:almost spherical}. Only simple manipulations are left to get a contradiction.

  At this stage we have two sets $S_*$ and $\widetilde{S}_*$ defined by \eqref{eq:stability set} and \eqref{eq:local min set}, such that if $\left(\frac{m}{\abs{B}}\right)^{1/n}\in \widetilde{S}$ then $B[m]$ is a $L^1$-local minimizer in \eqref{eq:isoperimetric problem},
  and if $\left(\frac{m}{\abs{B}}\right)^{1/n}\notin S_*$ then $B[m]$ is \emph{not} a $L^1$-local minimizer in \eqref{eq:isoperimetric problem}. We are left to study the variations of the functions
  \[\gamma \mapsto 1-\gamma^{1+\alpha}\frac{\mu_k^\alpha-\mu_1^\alpha}{\lambda_k-\lambda_1} +\gamma^{1+\beta} \frac{A \beta}{\lambda_k-\lambda_1},\quad k\geq 2\]
  to get the conclusions of the theorem. This is done in appendix \ref{ap:variation}.
\end{proof}

\section{Large volume minimizers for $\alpha<\beta$}\label{sec:large volume minimizers}
\subsection{Hausdorff convergence of large volume minimizers for ${\alpha <\beta}$}\label{subsec:hausdorff convergence}

Here we prove theorem \ref{thm:hausdorff boundary}, \textit{i.e.} that large volume minimizers of \eqref{eq:isoperimetric problem} are almost balls if $\alpha<\beta$. Note that if $\beta<1$, we know that large volume minimizers are not exactly balls. Indeed, in virtue of proposition \ref{prop:stability}, balls are not even local minimizers in this case.

The idea behind the proof is that if $\alpha<\beta$, then for a borel set $E\subset \rn$ of volume $m>0$
with $m$ large, the predominant term in $\eaba(E)$ is $\uba(E)$. This can be seen by rescaling:
\begin{equation}\label{eq:rescaling}
    \eaba(E)= \gamma^{n-1} \left(P(E^*)+\gamma^{1+\alpha}\V{E^*}+\gamma^{1+\beta}\uba(E^*)\right),
\end{equation}
where we have set $\gamma:=\left(\frac{m}{\abs{B}}\right)^{ \frac{1}{n}}$ and $E^*:=\frac{1}{\gamma}E$. As the unique volume constrained minimizer of $\uba$
is the ball $B[m]$, this implies that if $E$ is a minimizer of $\eaba$ at mass $m$ for $m$ large, $\uba(E)$ must be close to $\uba(B[m])$. This in turn will imply that $E$ is close to $B[m]$. Note that according to the rescaling \eqref{eq:rescaling}, proving theorem \ref{thm:hausdorff boundary} is equivalent to proving that if $(E_\gamma)_{\gamma >0}$ is a family of borel sets such that $\abs{E_\gamma}=\abs{B}$, and each set $E_\gamma$ is a volume-constrained minimizer of the functional
\begin{equation}\label{eq:rescaled functional}
  \eg:=P+\gamma^{1+\alpha}\mathcal{V}_\alpha+\gamma^{1+\beta}\uba,
\end{equation}
then the boundaries of the sets $(E_\gamma)$ Hausdorff-converge to the boundary of the unit ball $B$ as $\gamma \to +\infty$.
First we will show the following convergence in measure:
\begin{lemma}\label{EdeltaB}
  We have
  \[ \abs{E_\gamma \Delta B} \tend{\gamma}{\infty}0 .\]
\end{lemma}
\noindent We will need the following stability lemma for the potential energy $\uba$.
\begin{lemma}\label{quantitativeU}
  For any borel set $E\subset\rn$ of volume $\abs{B}$, we have \[\U{E}-\U{B} \geq
  \frac{A\beta}{8P(B)}\abs{E \Delta B}^2.\]
\end{lemma}

\begin{proof}
  Let $E\subset \rn$ be a borel set of volume $\abs{B}$. Define $r_1\geq 0$ and
  $r_2>0$ to be such that $\abs{\{x \in \rn \colon  r_1 \leq \abs{x} \leq 1
  \}}=\abs{\{x \in \rn \colon 1 \leq \abs{x} \leq r_2\}}=\abs{E \setminus
  B}=\abs{B\setminus E}$. Explicitely, $r_1=\left( 1-n \frac{ \abs{E \setminus
  B}}{P(B)} \right)^{ \frac{1}{n}}$ and $r_2=\left( 1+n \frac{ \abs{E \setminus
  B}}{P(B)} \right)^{ \frac{1}{n}}$. We then have
  \begin{align}\label{Ugap}
    \U{E}-\U{B}  &  = \int_{E \setminus B} A\abs{x}^\beta \mathrm{d}x-\int_{B \setminus
    E} A\abs{x}^\beta \mathrm{d}x \nonumber \\
    & \geq \int_{ \{x \in \rn \colon 1 \leq \abs{x} \leq r_2\}} A\abs{x}^\beta
    \mathrm{d}x-\int_{ \{x \in \rn \colon  r_1 \leq \abs{x} \leq 1 \}} A\abs{x}^\beta
    \mathrm{d}x \nonumber \\
    \intertext{(for $x \to \abs{x}^\beta$ is symmetric non-decreasing),}
    & =AP(B)\left( \int_{1}^{ r_2}r^\beta r^{n-1} \mathrm{d}r - \int_{ r_1}^1r^\beta
    r^{n-1} \mathrm{d}r \right) \nonumber \\
    & = \frac{AP(B)}{n+\beta} \left(  r_2^{n+\beta}-1-\left( 1-r_1^{n+\beta} \right)
    \right) \nonumber \\
    & = \frac{AP(B)}{n+\beta} \left( \left( 1+n \frac{ \abs{E \setminus B}}{P(B)}
    \right)^{ \frac{n+\beta}{n}}-1 \right. \nonumber \\
    &\qquad \left. {}-\left( 1-\left( 1-n \frac{ \abs{E \setminus
    B}}{P(B)} \right)^{ \frac{n+\beta}{n}} \right) \right).
  \end{align}
  Now, setting $\lambda:=\frac{n+\beta}{n}$ and $f(r):=\left( \left( 1+r
  \right)^{\lambda}-1-\left( 1-\left( 1-r \right)^{ \lambda}\right) \right)$, we have
  $f''(r)=\lambda \left( \lambda-1 \right)\left( (1+r)^{\lambda-2}+(1-r)^{\lambda-2}
  \right) \geq \lambda \left( \lambda-1 \right)$. As $f'(0)=f(0)=0$, we get $f(r) \geq
  \lambda (\lambda -1) \frac{r^2}{2}$, which yields the result with \eqref{Ugap}.
\end{proof}

Lemma \ref{EdeltaB} is then easily deduced from lemma \ref{quantitativeU} :

\begin{proof}[Proof of lemma \ref{EdeltaB}]
  We have
  \begin{align}
    \gamma^{1+\beta} \U{E_\gamma} &\leq \eg(E_\gamma) \nonumber \\
     &\leq \eg(B) \nonumber \\
     &=P(B)+\gamma^{1+\alpha}\V{B}+\gamma^{1+\beta}\U{B}, \nonumber
  \end{align}
  so
  \begin{equation*}
    \U{E_\gamma}-\U{B} \leq \frac{1}{\gamma^{1+\beta}}\left(
    P(B)+\gamma^{1+\alpha}\V{B} \right).
  \end{equation*}
  This implies $\U{E_\gamma}-\U{B}\tend{\gamma}{\infty}0$, which concludes the proof thanks to Lemma \ref{quantitativeU}.
\end{proof}

We are now in position to prove theorem \ref{thm:hausdorff boundary}.

\begin{proof}[Proof of theorem \ref{thm:hausdorff boundary}]
  \emph{Step one.}
  We show that given $R>1$, for $\gamma$ large enough we have $E_\gamma \subset B_{ R}$.

  Given $R>1$, set
  $F=\mu \left( E_\gamma \cap B_R \right)$, with $\mu>0$ such that $\abs{F}=\abs{B}$, ie
  $\mu=\left( \frac{ \abs{E_\gamma}}{ \abs{E_\gamma \cap B_R}} \right)^{ \frac{1}{n}}=\left(
  \frac{1}{1-u} \right)^{ \frac{1}{n}}$, where $u= \frac{ \abs{E_\gamma \setminus B_R}}{
  \abs{E_\gamma}}$. We have
  \begin{align}
    \eg(F)  & = \mu^{n-1}P(E_\gamma \cap B_R)+\mu^{n+\alpha}\gamma^{1+\alpha}\V{E_\gamma \cap
    B_R}\nonumber \\
    &\qquad {}+\mu^{n+\beta}\gamma^{1+\beta}\U{E_\gamma \cap B_R} \nonumber\\
    & \leq \mu^{n+\beta}\eg(E_\gamma \cap B_R).
  \end{align}
  Take $\eta>0$ to be chosen later, and then $\epsilon>0$ such that for all $v \in
  [0,\epsilon)$, $\left( \frac{1}{1-v} \right)^\frac{n+\beta}{n}\leq 1+\left(
  \frac{n+\beta}{n}+\eta \right)v$.
  According to Lemma \ref{EdeltaB}, if $\gamma$ has been taken large enough, we can assume
  that $u \leq \epsilon$, and so $\mu^{n+\beta} \leq1+\left( \frac{n+\beta}{n}+\eta
  \right)u$. Then using $P(E_\gamma \cap B_R) \leq P(E_\gamma)$, $\V{E_\gamma \cap B_R} \leq \V{E_\gamma}$ and
  $\U{E_\gamma}-\U{E_\gamma\cap B_R}\geq A\abs{E_\gamma\setminus B_R}R^\beta$, we find
  \begin{align}\label{egtruncation}
    \eg(F) & \leq \left( 1+\left( \frac{n+\beta}{n}+\eta \right)u \right)\eg(E_\gamma \cap B_R)
    \nonumber\\
    &  \leq \eg(E_\gamma)-\gamma^{1+\beta}A\abs{E_\gamma \setminus B_R}R^\beta+\left(
    \frac{n+\beta}{n}+\eta \right)u\eg(E_\gamma) \nonumber\\
    & = \eg(E_\gamma)+\left( \left( \frac{n+\beta}{n}+\eta \right)\eg(E_\gamma)-\gamma^{1+\beta}
    AR^\beta \abs{B} \right)u \nonumber \\
    & \leq \eg(E_\gamma) + \left( \left( \frac{n+\beta}{n}+\eta \right)\eg(B)-\gamma^{1+\beta}
    AR^\beta \abs{B} \right)u.
  \end{align}
  But as $\gamma \to \infty$,
  \begin{align}
    \eg(B)&=\gamma^{1+\beta} \U{B}+o(\gamma^{1+\beta}) \nonumber \\
    &=\gamma^{1+\beta}A\frac{1}{n+\beta}P(B)+o(\gamma^{1+\beta}) \nonumber\\
    &=\gamma^{1+\beta}A\frac{n}{n+\beta}\abs{B}+o(\gamma^{1+\beta}). \nonumber
  \end{align}
  So with \eqref{egtruncation},
  \begin{align}
    \eg(F)\leq \eg(E_\gamma)+\left( \gamma^{1+\beta}A\abs{B}(1+\frac{
    n}{n+\beta}\eta-R^{\beta})+o(\gamma^{1+\beta}) \right)u. \nonumber
  \end{align}
  Recall that $R>1$, so that if $\eta$ has been taken small enough, we get that for
  $\gamma$ large enough, $\eg(F) \leq \eg(E_\gamma)$, with equality if and only if $u$=0, i.e. $E_\gamma \subset B_R$.

  \emph{Step two.}
  We show that given $\delta>0$, for $\gamma$ large enough we have $B_{1-\delta}\subset E_\gamma$. This is done by taking some mass of $E_\gamma$ outside a
  certain ball $B_R$ and putting it in $E_\gamma\cap B_r$ for a well chosen $r$. In the proof we use lemma \ref{per_in_ball} below to show that the perimeter decreases under such a transformation for a well chosen $r \in (1-\delta,1).$
  On the other hand, the increase of $\mathcal{V}_\alpha$ is compensated by the decrease of $\uba$ if $\gamma$ has been taken large enough.

  Let us set $F= B \setminus E_\gamma$ and $\epsilon=\delta/2$. From lemma
  \ref{EdeltaB} we know that that if $\gamma$ has been taken large enough we have
  $\abs{F}<\abs{B}\left(\frac{\epsilon}{2}\right)^n$. Thus we can apply lemma
  \ref{per_in_ball} below with $r_0=1-\epsilon$, to get a $r\in(1-\delta,1-\epsilon)$ such that
  \begin{equation}
    P(F,B_r)\geq \mathcal{H}^{n-1}(F\cap \partial B_r).
  \end{equation}
  As $\abs{E_\gamma}=\abs{B}$, we have $\abs{E_\gamma \setminus B}=\abs{B \setminus E_\gamma} \geq \abs{B_r
  \setminus E_\gamma}$, so there exists $R \geq 1$ such that $\abs{E_\gamma \setminus B_R}=\abs{B_r
  \setminus E_\gamma}$. Now let us set
  \[E_\gamma'=(E_\gamma \cap B_R)\cup B_r,\]
  and compare $\eg(E_\gamma')$ and $\eg(E_\gamma)$.
  Using classical formulae for the perimeter of the union or the intersection of a set
  with a ball (see \cite[remark 2.14]{Giusti}), we have
  \begin{align*}
    P(E_\gamma')&=\mathcal{H}^{n-1}(E_\gamma\cap \partial B_R)+P(E_\gamma,\overline{B_r}^c\cap
    B_R)+\mathcal{H}^{n-1}(E_\gamma^c\cap \partial B_r), \\
    P(E_\gamma)&=P(E_\gamma,B_R^c)+P(E_\gamma,\overline{B_r}^c\cap B_R)+P(E_\gamma,\overline{B_r}),
  \end{align*}
  so that
  \begin{multline}\label{eq:deltaP}
    P(E_\gamma')-P(E_\gamma)=\mathcal{H}^{n-1}(E_\gamma\cap \partial
    B_R)-P(E_\gamma,B_R^c) \\ +\mathcal{H}^{n-1}(E_\gamma^c\cap \partial B_r)-P(E_\gamma,\overline{B_r}).
  \end{multline}
  From the classical inequality $P(E_\gamma\cap B_R) \leq P(E_\gamma)$, we get that
  $\mathcal{H}^{n-1}(E_\gamma\cap \partial B_R)\leq P(E_\gamma,B_R^c)$, so \eqref{eq:deltaP} gives
  \[P(E_\gamma')-P(E_\gamma)\leq\mathcal{H}^{n-1}(E_\gamma^c\cap \partial B_r)-P(E_\gamma,\overline{B_r}).\]
  But
  \[\mathcal{H}^{n-1}(E_\gamma^c\cap \partial B_r)=\mathcal{H}^{n-1}(E_\gamma^c\cap B\cap \partial
  B_r)=\mathcal{H}^{n-1}((B\setminus E_\gamma) \cap \partial B_r),\]
  and
  \[P(E_\gamma,\overline{B_r}) = P(E_\gamma^c,\overline{B_r})= P(E_\gamma^c\cap B,\overline{B_r}) = P(B
  \setminus E_\gamma,\overline{B_r}),\]
  So, recalling that $F=B\setminus E_\gamma$, we obtain
  \begin{align*}
    P(E_\gamma')-P(E_\gamma)&\leq\mathcal{H}^{n-1}(F\cap \partial B_r)-P(F,\overline{B_r}) \\
    &\leq\mathcal{H}^{n-1}(F\cap \partial B_r)-P(F,B_r),
  \end{align*}
  so by the choice of $r$,
  \begin{equation}\label{eqPdecrease}
    P(E_\gamma')\leq P(E_\gamma).
  \end{equation}
  Now we estimate the variation of $\mathcal{V}_\alpha$. Let us define the non-local potential:
  \begin{equation}\nonumber
    \Phi_E^\alpha(x)=\int_E \frac{\mathrm{d}x}{ \abs{x-y}^{n-\alpha}}.
  \end{equation}
  With this notation, we have
  \begin{align*}
    \V{E_\gamma'}-\V{E_\gamma}&=\int_{E_\gamma'}\Phi_{E_\gamma'}^\alpha-\int_{E_\gamma}\Phi_{E_\gamma}^\alpha \\
    &=\int_{E_\gamma'}\Phi_{E_\gamma'}^\alpha-\int_{E_\gamma}\Phi_{E_\gamma'}^\alpha+\int_{E_\gamma'}\Phi_{E_\gamma}^\alpha-\int_{E_\gamma}\Phi_{E_\gamma}^\alpha
    \\
    &=\int_{E_\gamma'\setminus E_\gamma}\Phi_{E_\gamma'}^\alpha-\int_{E_\gamma\setminus
    E_\gamma'}\Phi_{E_\gamma'}^\alpha+\int_{E_\gamma'\setminus E_\gamma}\Phi_{E_\gamma}^\alpha-\int_{E_\gamma\setminus
    E_\gamma'}\Phi_{E_\gamma}^\alpha \\
    &\leq 4 \sup_{\abs{F}=\abs{B}} \|\Phi_F^\alpha\|_\infty \abs{E_\gamma \setminus E_\gamma'}.
  \end{align*}
  So using the simple lemma \ref{lemmaregularity_potential},
  \begin{equation}\label{eq:deltaV}
    \V{E_\gamma'}-\V{E_\gamma}\leq C(n,\alpha)\abs{E_\gamma\setminus B_r}.
  \end{equation}
  As for $\uba$, we have
  \begin{align*}
    \uba(E_\gamma')-\uba(E_\gamma)&=\int_{B_r \setminus E_\gamma} A \abs{x}^{\beta} \mathrm{d}x-\int_{E_\gamma
    \setminus B_R} A \abs{x}^{\beta} \mathrm{d}x\\
    &\leq \int_{B_r \setminus E_\gamma} A r^{\beta} \mathrm{d}x-\int_{E_\gamma \setminus B_R} A
    R^{\beta} \mathrm{d}x \\
    &\leq \int_{B_r \setminus E_\gamma} A (1-\epsilon)^{\beta} \mathrm{d}x-\int_{E_\gamma \setminus
    B_R} A  \mathrm{d}x \\
    &=A\left((1-\epsilon)^{\beta}\abs{B_r \setminus E_\gamma}-\abs{E_\gamma\setminus B_R}\right) \\
    &=A\left((1-\epsilon)^{\beta}-1\right)\abs{E_\gamma\setminus B_r}.
  \end{align*}
  This last estimate with \eqref{eqPdecrease} and \eqref{eq:deltaV} gives
  \[\eg(E_\gamma')-\eg(E_\gamma) \leq
  \left(\gamma^{1+\alpha}C(n,\alpha)+\gamma^{1+\beta}A\left((1-\epsilon)^{\beta}-1\right)\right)\abs{E_\gamma\setminus B_r}.\]
  As $E_\gamma$ is a minimizer, we have $\eg(E_\gamma')-\eg(E_\gamma)\geq 0$, so for $\alpha<\beta$ and
  $\gamma$ large enough (depending only of $n$, $\alpha$, $\beta$, $A$, $\delta$), this last inequality implies
  \[\abs{B_r \setminus E_\gamma}=0,\quad \textit{i.e.} \quad B_r\subset E_\gamma.\]
  This concludes \emph{Step two}.

  The theorem is just \emph{Step one} and \emph{Step two} together.
\end{proof}

\begin{rmk}
  With this proof, we see that the result of theorem \ref{thm:hausdorff boundary} is also valid for any $\alpha \in (0,n)$ and $\beta>0$ if, instead of letting the mass $m$ go to $+\infty$, we let the quantity $A\gamma^{\beta - \alpha}$ go to $+\infty$ (with $\gamma=\left(\frac{m}{\abs{B}}\right)^{\frac 1n}$).
\end{rmk}

\begin{lemma}\label{per_in_ball}
  Given $F\subset  \rn$ a set of finite perimeter, $r_0>0$, and $\epsilon>0$, assume
  that
  \begin{equation}\label{eqvolume}
    \abs{F}\leq\abs{B}\left(\frac{\epsilon}{2}\right)^n.
  \end{equation}
  Then there exists $r \in (r_0-\epsilon,r_0)$ such that,
  \begin{equation}
    P(F,B_r)\geq \mathcal{H}^{n-1}(F\cap \partial B_r).
  \end{equation}
\end{lemma}
\begin{proof}
  We argue by contradiction and assume that \eqref{eqvolume} holds, and
  \begin{equation}
    \forall r \in (r_0-\epsilon,r_0), \; P(F,B_r)< \mathcal{H}^{n-1}(F\cap \partial
    B_r).\nonumber
  \end{equation}
  Adding $\mathcal{H}^{n-1}(F\cap \partial B_r)$ to both sides, this is equivalent to
  \begin{equation*}
    P(F\cap B_r)<2\mathcal{H}^{n-1}(F\cap \partial B_r).
  \end{equation*}
  Using the isoperimetric inequality we get
  \begin{equation}\label{ineqper_in_ball}
    \left(\frac{\abs{F\cap
    B_r}}{\abs{B}}\right)^{\frac{n-1}{n}}P(B)<2\mathcal{H}^{n-1}(F\cap \partial B_r).
  \end{equation}
  Now set for $r\geq 0$, $f(r)=\abs{F\cap\ B_r}$. We can assume $f(r) \neq 0$ for all
  $r \in (r_0-\epsilon,r_0)$ otherwise the lemma is trivially true. We have for almost
  all $r\in(0,\infty)$,
  \[f'(r)=\mathcal{H}^{n-1}(F\cap \partial B_r).\]
  Thus \eqref{ineqper_in_ball} gives for almost all $r \in (r_0-\epsilon,r_0)$,
  \[\frac{1}{n}f'(r)f(r)^{\frac{1}{n}-1}>\frac{P(B)}{2n\abs{B}^{\frac{n-1}{n}}}=\frac{\abs{B}^{\frac{1}{n}}}{2}.\]
  Integrating on the interval $(r_0-\epsilon,r_0)$, we get
  \[f(r_0)^{\frac{1}{n}}-f(r_0-\epsilon)^{\frac{1}{n}}>\frac{\epsilon \abs{B}^{\frac{1}{n}}}{2},\]
  so
  \[f(r_0)^{\frac{1}{n}}>\frac{\epsilon \abs{B}^{\frac{1}{n}}}{2},\]
  which contradicts \eqref{eqvolume}.
\end{proof}


\subsection{Large volume minimizers = balls for $\alpha<\beta$ and $\beta >1$}\label{subsec:exactly balls}
Here we prove theorem \ref{thm:ball unique minimizer}, \textit{i.e.} that if we assume in addition that $\alpha>1$, then large volume minimizers
are exactly balls. We conjecture that the theorem is also true when $\alpha \in (0,1]$, as long as $\beta>1$.
For $\beta<1$, it cannot be true as we know from proposition \ref{prop:stability} that for $m$ large the ball $B[m]$ is not even a local minimizer.
Note that in dimension $1$, using theorem \ref{thm:hausdorff boundary}, one can perform some computations to show that the theorem is indeed true under the more general assumption $\beta>\max(1,\alpha)$.

The proof relies heavily on the following simple lemma:
\begin{lemma}\label{lemmaregularity_potential}
  If $\alpha>1$, then there exists a constant $C(n,\alpha)>0$ such that for any set $E \subset \rn$ of volume $\abs{E}=\abs{B}$, we have
  \[\|\Phi_E^\alpha\|_{C^1(\rn)} \leq C(n,\alpha), \quad \text{where} \quad \Phi_E^\alpha(x)=\int_E \frac{\mathrm{d}x}{ \abs{x-y}^{n-\alpha}}.\]
\end{lemma}
\noindent This lemma is not true as soon as $\alpha \leq 1$, where we just get $\alpha$-Hölder
continuity instead of Lipschitz continuity. We refer to \cite{bonacini_cristoferi} for a proof.

\begin{proof}[Proof of theorem \ref{thm:ball unique minimizer}]
  Rescaling the functional as usual, we need to show that if $E\subset \rn$ is such that $\abs{E}=\abs{B}$, and $E$ is a volume-constrained minimizer of $\eg$ (see \eqref{eq:rescaled functional}), then $E$=$B$. Let us show that for $\gamma>0$ large enough, we have
  \begin{equation}\label{eq:U+B_decrease}
    \gamma^{1+\alpha}\V{E}+\gamma^{1+\beta}\uba(E)\geq \gamma^{1+\alpha}\V{B}+\gamma^{1+\beta}\uba(B).
  \end{equation}
  The theorem will then result from the isoperimetric inequality: $P(E) > P(B)$ if $E\neq B$. We divide the proof of \eqref{eq:U+B_decrease} into two steps. In \textit{step one} we compare $E$ to the subgraph of a function over the sphere, by concentrating the mass of $E$ on each half line through the origin. In \textit{step two}, we show that \eqref{eq:U+B_decrease} holds for subgraphs of sufficiently small functions over the sphere.

  \noindent\textit{Step one}. For any $x\in \s$, define $u(x)\in \R$ by the equation
  \begin{equation}
    \int_0^{1+u(x)}r^{n-1}\mathrm{d}r=\int_{\R_+} \1_{rx\in E}r^{n-1}\mathrm{d}r.
  \end{equation}
  Then set
  \begin{equation}
    E_u=\{t(1+u(x)), \;t\in[01),\; x\in \s \}.
  \end{equation}
  We have
  \begin{align}
    \abs{E_u}&=\int_{\s} \int_0^{1+u(x)}r^{n-1}\mathrm{d}r \hau(x) \\
    &=\int_{\s}  \int_{\R_+} \1_{rx\in E}r^{n-1}\mathrm{d}r \hau(x) \\
    &=\abs{E},
  \end{align}
  thus $E_u$ satisfies the volume constraint. Now we estimate the variation of $\uba$. From theorem \ref{thm:hausdorff boundary} we know that, taking $\gamma$ large enough, we can assume $B_{ \frac{1}{2}}\subset E$. Thus we have
  \begin{align}
    \uba(E_u)-\uba(E) &= \int_{E_u \setminus B_{ \frac{1}{2}}} A \abs{x}^{\beta}\mathrm{d}x-\int_{E \setminus B_{ \frac{1}{2}}} A \abs{x}^{\beta}\mathrm{d}x \nonumber\\
    & = \int_{\s}\int_{ (\frac{1}{2},u(x))} A r^{\beta}r^{n-1}\hau(x)\nonumber\\
    &\qquad{}-\int_{\s}\int_{ (\frac{1}{2},\infty)}\1_{rx\in E} A r^{\beta}r^{n-1}\hau(x) \nonumber\\
    &= \int_{\s}\left(\int_{ (\frac{1}{2},u(x))} A r^{\beta}r^{n-1}\mathrm{d}r-\int_{ E_x} A r^{\beta}r^{n-1}\mathrm{d}r\right)\hau(x), \label{eq:deltaUsubgraph}
  \end{align}
  where we have set
  \[E_x:=\{r \geq \frac{1}{2} : rx \in E\}.\]
  Here we need a simple lemma from optimal transportation on the real line.
  \begin{lemma}\label{lemma:optimal transportation}
    Given a measurable set $S \subset (1/2,\infty)$ such that ${\int_{S}r^{n-1}\mathrm{d}r<\infty}$, let $u>0$ be such that
    \[ \int_{S}r^{n-1}\mathrm{d}r=\int_{ (\frac{1}{2},u)} r^{n-1}\mathrm{d}r.\]
    Then there exists a measurable map $T:(1/2,u) \to (1/2,\infty)$ such that
    \[ \1_Sr^{n-1}\mathrm{d}r=T\#(\1_{(1/2,u)}r^{n-1}\mathrm{d}r), \]
    \textit{i.e.} for any non-negative measurable function $f$,
    \[\int_{S}f(r)r^{n-1}\mathrm{d}r=\int_{(1/2,u)}f(T(r))r^{n-1}\mathrm{d}r.\]
    What is more we have \[\forall r \in (1/2,u),\; T(r) \geq r.\]
  \end{lemma}
  The existence of the map $T$ is a consequence of the existence of an optimal transport map for non-atomic probability measures on the real line. For each $x \in \s$, we apply this lemma to $S=E_x$, to get a corresponding map $T_x$. Then \eqref{eq:deltaUsubgraph} becomes
  \begin{equation}\label{eq:deltaUsubgraph_bis}
    \uba(E_u)-\uba(E) = \int_{\s}\left(\int_{ (\frac{1}{2},u(x))} \left(A r^{\beta}-AT_x(r)^{\beta}\right)  r^{n-1}\mathrm{d}r\right)\hau(x).
  \end{equation}
  Now let us compute the variation of the Riesz energy $\mathcal{V}_\alpha$ in a similar fashion~:
  \begin{align}
    \V{E_u}-\V{E}&=\int_{E_u}\Phi_{E_u}^\alpha-\int_{E}\Phi_{E}^\alpha \nonumber\\
    &=\int_{E_u}\Phi_{E_u}^\alpha+\int_{E_u}\Phi_{E}^\alpha-\int_{E}\Phi_{E_u}^\alpha-\int_{E}\Phi_{E}^\alpha \nonumber \\
    &=\int_{E_u}(\Phi_{E_u}^\alpha+\Phi_{E}^\alpha)-\int_{E}(\Phi_{E_u}^\alpha+\Phi_{E}^\alpha) \nonumber \\
    &=\int_{\s} \int_{(1/2,u(x))}\left[(\Phi_{E_u}^\alpha+\Phi_{E}^\alpha)(rx)\right. \nonumber\\
    &\qquad\left. {}-(\Phi_{E_u}^\alpha+\Phi_{E}^\alpha)(T_x(r)x) \right]r^{n-1}\mathrm{d}r\hau(x). \label{eq:deltaVsubgraph}
  \end{align}
  To estimate \eqref{eq:deltaUsubgraph_bis} and \eqref{eq:deltaVsubgraph}, we use the two following inequalities:

  \[\forall x\in \s, \forall s\geq r>\frac{1}{2},\; r^\beta-s^\beta \leq -C(\beta)\abs{r-s}, \]
  \[ \quad (\Phi_{E_u}^\alpha+\Phi_{E}^\alpha)(rx)-(\Phi_{E_u}^\alpha+\Phi_{E}^\alpha)(sx) \leq C(n,\alpha) \abs{r-s}, \]
  where the second inequality comes from lemma \ref{lemmaregularity_potential}. With these and \eqref{eq:deltaUsubgraph_bis} and \eqref{eq:deltaVsubgraph}, we get
  \begin{align*}
    &\left(\gamma^{1+\alpha}\V{E_u}+\gamma^{1+\beta}\uba(E_u)\right) - \left(\gamma^{1+\alpha}\V{E}+\gamma^{1+\beta}\uba(E)\right) \\
    &\leq\int_{\s} \int_{(1/2,u(x))}\left(\gamma^{1+\alpha}C(n,\alpha)-\gamma^{1+\beta}AC(\beta)\right)\abs{r-T_x(r)} r^{n-1}\mathrm{d}r\hau(x).
  \end{align*}
  From this inequality we get that if $\gamma$ is large enough (depending only on $n$, $\alpha$, $\beta$, $A$), then
  \begin{equation}\label{eq:U+V_1}
    \gamma^{1+\alpha}\V{E_u}+\gamma^{1+\beta}\uba(E_u)\leq \gamma^{1+\alpha}\V{E}+\gamma^{1+\beta}\uba(E).
  \end{equation}

  \noindent\textit{Step two}. We show that there exists $\epsilon=\epsilon(n,\alpha,\beta,A)>0$, such that for any $\gamma$ large enough, if $\|u\|_{L^\infty(\s)} <\epsilon$, then
  \begin{equation}\label{eq:U+V_2}
    \gamma^{1+\alpha}\V{B}+\gamma^{1+\beta}\uba(B)\leq \gamma^{1+\alpha}\V{E_u}+\gamma^{1+\beta}\uba(E_u).
  \end{equation}
  Remark that by theorem \ref{thm:hausdorff boundary}, the condition $\|u\|_{L^\infty(\s)} <\epsilon$ is satisfied if $\gamma$ has been taken large enough.
  The inequality \eqref{eq:U+V_2} will result from this computational lemma, whose proof is postponed :
  \begin{lemma}\label{lemmataylorexp_u_v}
    Given a measurable function $u:\s \to \R$ with $\|u\|_{L^\infty(\s)} <1$, set for $t\geq 0$
    \[E_t:=\{s(1+tu(x))x, ~ x\in \s,~ s\in[0,1) \}.\]
    Assume that $\abs {E_t}=\abs{B}$. Then for $t$ small enough, depending only on the dimension $n$, we have
    \begin{equation} \label{eqtaylorexp_u}\U{E_t} \geq \U{B}+A\beta \frac{t^2}{2}\nldc{u}-C(n,\beta)t^3\nldc{u},
    \end{equation}
    and
    \begin{align}
      \V{E_t} &\geq \V{B}-\frac{t^2}{2}\left([u]^2_{ \frac{1-\alpha}{2}}-\alpha(n+\alpha) \nldc{u}\right)\nonumber \\
      &\qquad{}-C(n)t^3\left([u]^2_{ \frac{1-\alpha}{2}}+\alpha\V{B}\nldc{u}\right),\label{eqtaylorexp_v}
    \end{align}
    where
    \[[u]^2_{ \frac{1-\alpha}{2}}=\int_{\s \times \s} \frac{\abs{u(x)-u(y)}^2 }{\abs{x-y}^{n-\alpha}} \hau(x)\hau(y).\]
  \end{lemma}
  \noindent Indeed for $\alpha>1$, we have
  \begin{align*}
    [u]^2_{ \frac{1-\alpha}{2}}&\leq \int_{\s \times \s} \frac{2(\abs{u(x)}^2+\abs{u(y)})^2 }{\abs{x-y}^{n-\alpha}} \hau(x)\hau(y)\\
    &=4\int_{\s \times \s} \frac{\abs{u(x)}^2}{\abs{x-y}^{n-\alpha}} \hau(x)\hau(y) \\
    &=4\int_{\s}\left(\int_{\s} \frac{\hau(y)}{\abs{x-y}^{n-\alpha}}\right)\abs{u(x)}^2 \hau(x) \\
    &=C(n,\alpha)\int_{\s} \abs{u(x)}^2 \hau(x) \\
    &=C(n,\alpha)\nldc{u},
  \end{align*}
  so that \eqref{eqtaylorexp_v} gives
  \[\V{E_t} \geq \V{B}-\frac{t^2}{2}C(n,\alpha)\nldc{u}-C(n,\alpha)t^3\nldc{u}.\]
  This implies that for $t$ small enough, depending only on $n$ and $\alpha$, we have
  \[\V{E_t} \geq \V{B}-t^2C(n,\alpha)\nldc{u}.\]
  Likewise, we get from \eqref{eqtaylorexp_u} that for $t$ small enough, depending only on $n$, $\beta$ and $A$, we have
  \[\U{E_t} \geq \U{B} +t^2C(n,\beta,A)\nldc{u}.\]
  These last two inequalities imply that there exists $\epsilon=\epsilon(n,\alpha,\beta,A)>0$, such that if $\|u\|_{L^\infty(\s)} <\epsilon$, then
  \begin{multline*}\gamma^{1+\alpha}\V{B}+\gamma^{1+\beta}\uba(B)\leq \gamma^{1+\alpha}\V{E_u}+\gamma^{1+\beta}\uba(E_u)\\
    +\left(\gamma^{1+\alpha}C(n,\alpha,\beta,A)\nldc{u}-\gamma^{1+\beta}C(n,\alpha,\beta,A)\nldc{u}\right),
  \end{multline*}
  which in turn implies \eqref{eq:U+V_2} for $\gamma$ large enough.

  The estimate \eqref{eq:U+B_decrease} is now a consequence of \eqref{eq:U+V_1} and \eqref{eq:U+V_2}. The theorem results from \eqref{eq:U+B_decrease} and the isoperimetric inequality.
\end{proof}

\begin{proof}[Proof of lemma \ref{lemmataylorexp_u_v}]
  The proof of \eqref{eqtaylorexp_v} is given in \cite[equation (5.20)]{vincent}, under the hypothesis ${\ncu{u}\leq 1}$ instead of $\|u\|_{L^\infty(\s)}\leq 1$. However it is clear from the proof that it holds also for $\|u\|_{L^\infty(\s)}\leq 1$ only. (The reason why it was stated with the stronger hypothesis $\ncu{u}\leq 1$ is because it is needed to get the corresponding estimate for the perimeter.)

  Let us prove \eqref{eqtaylorexp_u}. Using spherical coordinates, we can compute
  \begin{align}
    \U{E_t} & =\int_{\s}\int_{0}^{1+tu(x)}A\abs{rx}^\beta r^{n-1}\mathrm{d}r\hau (x) \nonumber \\
    & =\int_{\s}\int_{0}^1A(1+tu(x))^{n+\beta}r^{n+\beta-1}\mathrm{d}r\hau (x) \nonumber\\
    & =\int_{\s}A\frac{(1+tu)^{n+\beta}}{n+\beta}\hau. \nonumber
  \end{align}
  Setting $h(t):=\int_{\s}(1+tu)^{n+\beta} \hau$, we then have $\U{E_t}-\U{B}=\frac{A}{n+\beta}(h(t)-h(0))$. Let us proceed to a Taylor expansion of $h$. We have
  \[(1+tu)^{n+\beta}\geq1+(n+\beta)tu+(n+\beta)(n+\beta-1)\frac{(tu)^2}{2}-C(n,\beta)\frac{(tu)^3}{3},\]
  So
  \begin{align}
    \frac{1}{n+\beta}(h(t)-h(0))  & \geq \int_{\s} tu\hau +(n+\beta-1)\int_{\s} \frac{(tu)^2}{2}\hau-C(n,\beta)t^3 \int_{\s} u^3\hau \nonumber \\
    &\geq \int_{\s} tu\hau +(n+\beta-1)\int_{\s} \frac{(tu)^2}{2}\hau -C(n,\beta)t^3 \nldc{u}. \label{taylorexph}
  \end{align}
  Now we use the volume constraint $\abs{E_t}=\abs{B}$ to estimate $\int tu$.
  The volume constraint can be expressed as
  \[\int_{\s} (1+tu)^n\hau=\int_{\s} 1\hau,\]
  and so
  \begin{align}
    \int_{\s}tu\hau &=\int_{\s}\left(tu-\frac{1}{n}\left((1+tu)^n-1\right)\right)\hau \nonumber\\
    &=-\sum_{k=2}^{n}\frac{1}{n}\binom{n}{k}\int_{\s}(tu)^k\hau \nonumber \\
    &\geq -\frac{n-1}{2} \int_{\s} (tu)^2 \hau-C(n)t^3 \nldc{u}
  \end{align}
  This with \eqref{taylorexph} gives \eqref{eqtaylorexp_u}.
\end{proof}


\section{Numerical minimization}\label{sec:numerical}
In this section we present our method and results for the numerical minimization
of the variational problem \eqref{eq:isoperimetric problem}, the constant $A\geq0$ being potentially $0$. In particular we apply this method with $A=0$ to give a numerical answer to the two questions raised at the end of the introduction.

\subsection{Method of the numerical minimization}\label{subsec:method}
We present a series of three modifications of the variational problem \eqref{eq:isoperimetric problem}
to arrive at a finite dimensional variational problem that can be easily numerically solved. All steps are justified by a $\Gamma$-convergence and compactness result. We refer to \cite{gamma_convergence} for definition and properties of $\Gamma$-convergence.

\emph{Step one} is standard when dealing with the perimeter. We use the classical Modica-Mortola theorem to relax the functional on sets, \text{i.e.} charateristic functions, into a functional on functions taking values in $[0,1]$. This allows us to use the vector space structure of functions and, after disctretization (\emph{step three}), usual optimization tools for functionals on $\R^d$.

\emph{Step two} is the key step for dealing with the non-local term $\mathcal{V}_\alpha$. We replace the ambient space $\rn$ with a large square with periodic boundary conditions, whose size is a new relaxation parameter. Then we can approximate the non-local term $\mathcal{V}_\alpha$ by a simple expression in Fourier variable.

In \emph{step three}, we discretize the problem by considering only trigonometric functions with frequencies lower than some integer $N$, and by computing the integral terms with riemann sums.
\begin{terminology}
  We say that a family of functionals $(\mathcal{F}_\epsilon)_{\epsilon>0}$ defined
  on a metric space $X$ enjoys property (C) (for compactness) if for any
  family $(u_\epsilon)_{\epsilon>0}$ of elements of $X$ such that $(F_\epsilon(u_\epsilon))_{\epsilon>0}$
  is bounded, there is a subsequence of $(u_\epsilon)_{\epsilon>0}$ that converges in $X$.
\end{terminology}
\noindent If a family of functionals $(\mathcal{F}_\epsilon)_{\epsilon>0}$ enjoys property
(C) and $\Gamma$-converges to a limit functional $\mathcal{F}$ when $\epsilon$ goes to $0$, then we know that for $\epsilon$ small enough, minimizers of $\mathcal{F}_\epsilon$ are close to minimizers of $\mathcal{F}$. Let us now describe and justify each step precisely.

\emph{Step one.} We use the classical Modica-Mortola theorem to replace
this problem on subsets of $\rn$, \text{i.e.} functions taking only values $0$ or $1$,
with a problem on functions taking any value between $0$ and $1$. More precisely, given a (large) smooth open bounded set $\Omega$ and a (small) $\epsilon>0$, we define the set $X$, and the functionals $\mathcal{F}_\epsilon:X \to \overline{\R}$ and $\mathcal{F}:X \to \overline{\R}$ by
\begin{equation*}
  X:=\{u \in L^1(\Omega,[0,1]) : \int u = m \},
\end{equation*}
\begin{align}
  \mathcal{F}_\epsilon(u) &=
  \begin{cases}
    \epsilon\int_{\rn} \abs{\nabla u}^2+\frac{1}{\epsilon} \int_{\rn} W(u)+\V{u} + A\int_{\rn} u(x) \abs{x}^\beta \mathrm{d}x \\
    \text{if } u \in H^1(\rn),\\
    +\infty\\
    \text{otherwise},
    \end{cases} \\
    \mathcal{F}(u) &=
    \begin{cases}
      P(\{u=1\})+\V{u} + A\int_{\rn} u(x) \abs{x}^\beta \mathrm{d}x \\
      \text{if $u$ only takes values $0$ and $1$,}\\
      +\infty \\
      \text{otherwise,}
    \end{cases}
  \end{align}
  where we have used the natural notation $\V{u}=\int_{\rn \times \rn} \frac{u(x)u(y)}{\abs{x-y}^{n-\alpha}}\mathrm{d}x\mathrm{d}y$,
  and $W$ is the following double well potential on $[0,1]$: $W(x)=x(1-x)$.
  Then from the Modica-Mortola theorem and the fact that the two last terms of the functionals $\mathcal{F}_\epsilon$ and $\mathcal{F}$ are continuous on $X$, we have
  \begin{equation*}
    \mathcal{F}_\epsilon \xrightarrow[\epsilon \to 0]{\Gamma}\mathcal{F},
    \quad \text{and $(\mathcal{F}_\epsilon)$ enjoys property (C).}
  \end{equation*}
  Note that considering functions on a bounded open set $\Omega$ is not restrictive provided that $\Omega$ is large enough, as minimizers of \eqref{eq:isoperimetric problem} are necessarily bounded.

  \emph{Step two.} We wish to reduce the domain to a (large) square with periodic boundary
  conditions, \textit{i.e.} a torus. Indeed, the non-local repulsive term has a simple expression in Fourier variable :
  \begin{equation}\label{eq:valpha_fourier}
    \V{u}=\frac{C(n,\alpha)}{(2\pi)^n}\int_{\rn} \abs{\xi}^{-\alpha}\abs{\Hat{u}(\xi)}^2\mathrm{d}\xi,
  \end{equation}
  with $\Hat{u}$ the Fourier transform of $u$ and
  $C(n,\alpha):=\frac{2^\alpha\pi^{\frac{n}{2}}\Gamma(\frac{\alpha}{2})}{\Gamma(\frac{n-\alpha}{2})}$,
  and $\Gamma$ the usual gamma function. This can be seen by noting that
  $\V{u}=\int u \mathcal{I}_\alpha(u)$ with $\mathcal{I}_\alpha(u)$ the Riesz
  potential of u, and using the Fourier expression of the Riesz potential (see
  \cite[Part V]{singular}). Thus we will approximate $\V{u}$ by
  \begin{equation}\label{eq:vat}
    \vat(u):= \frac{C(n,\alpha)}{T^n}\sum_{k\in \Z^n \setminus \{0\}}\abs{\frac{2k\pi}{T}}^{-\alpha}\abs{c_{k,T}(u)}^2,
  \end{equation}
  where $c_{k,T}(u):=\int_{[-T/2;T/2]^n} u(x) e^{\frac{-2ik\pi x}{T}}\mathrm{d}x$ is
  the $k$-th Fourier coefficient of $u$ on $[-T/2;T/2]^n$, for some (large) $T>0$.
  More precisely, let us define the functional $\mathcal{F}_{\epsilon,T} : X \to \overline{\R}$ by
  \begin{equation*}
    \mathcal{F}_{\epsilon,T}(u)=
    \begin{cases}
      \epsilon\int\abs{\nabla u}^2+\frac{1}{\epsilon} \int W(u)+\vat(u) + A\int u(x) \abs{x}^\beta \mathrm{d}x \\
      \text{if}\quad u\in H^1(\rn),\\
      +\infty\\
      \text{otherwise}.
    \end{cases}
  \end{equation*}
  Then we have
  \begin{equation}\label{eq:grandtore}
    \mathcal{F}_{\epsilon,T} \xrightarrow[T \to 0]{\Gamma}\mathcal{F_\epsilon},
    \quad \text{and $(\mathcal{F}_{\epsilon,T})_{T>0}$ enjoys property (C).}
  \end{equation}
  We omit the proof of \eqref{eq:grandtore}, as it presents no major difficulty. It relies mostly on convergence of Riemann sums. However, we emphasize the following remark:
  \begin{rmk}
    For property $(C)$ to be valid, it is necessary to assume that all functions are supported in a given bounded set $\Omega$ (see section \ref{subsec:numerical results} for further comments).
  \end{rmk}

\emph{Step three.} As the final step, we discretize the variational problem. Let us first extend $\mathcal{F}_{\epsilon,T}$ to the functions $u\in H^1([-T/2;T/2]^n)$ that are not supported in $\Omega$ by setting $\mathcal{F}_{\epsilon,T}(u)=+\infty$ in this case. For
$N\in 2 \N$ large, instead of considering the whole space $H^1([-T/2;T/2]^n)$, we only consider the space
\begin{multline}E_N = \{u \in \mathrm{Vect} (e^{\frac{2i\pi}{T}k \cdot x})_{k \in \{-\frac{N}{2}+1,...,\frac{N}{2}\}^n} : \\ \forall j \in \{-\frac{N}{2}+1,\ldots,\frac{N}{2}\}^n, u(jT/N) \in [0,1], \\ u(jT/N)=0 \quad\text{if}\quad jT/N \notin \Omega, \quad\text{and}\quad \int u=m\}.
\end{multline}
For $u \in E_N$, we set
\[\mathcal{W}_N(u) = (T/N)^n\sum\limits_{j \in \{-N/2+1,...,N/2\}^n}  W(u(jT/N)),\]
and
\[\mathcal{U}_{\beta,A,N}(u) = A(T/N)^n\sum\limits_{j \in \{-N/2+1,...,N/2\}^n}  u(jT/N))\abs{jT/N}^{\beta}.\]
Then we define the functional $\mathcal{F}_{\epsilon,T,N} : H^1([-T/2;T/2]^n) \to \overline{\R}$ by
\begin{equation*}
  \mathcal{F}_{\epsilon,T,N}(u)=
  \begin{cases}
    \epsilon\int \abs{\nabla u}^2+\frac{1}{\epsilon} \mathcal{W}_N(u)+\vat(u) + \mathcal{U}_{\beta,A,N} \\
    \text{if}\quad u \in E_N,\\
    +\infty\\
    \text{otherwise}.
  \end{cases}
\end{equation*}
We have in the sense of the weak $H^1$ topology,
\begin{equation}\label{eq:discretize}
  \mathcal{F}_{\epsilon,T,N} \xrightarrow[N \to \infty]{\Gamma}\mathcal{F}_{\epsilon,T} \quad \text{and $(\mathcal{F}_{\epsilon,T,N})$ enjoys property (C).}
\end{equation}

In the proof of \eqref{eq:discretize}, we will use the following technical lemma, which shows that a triogonometric function whose frequencies are lower than $N$ is well represented by its values on a grid with step size $1/N$.

\begin{lemma}\label{lemma:riemann sum trigo}
  Let $(u_N)$ be a converging sequence in $L^2([0;1]^n)$, such that for every
  $N\geq 0$, $u_N \in E_N$. Then for any bounded uniformly continuous functions $\phi : \R \to \R$ and $\psi : [0,1]^n \to \R$, we have
  \begin{equation}\nonumber \label{eq:riemann sum trigo}
  \abs{\frac{1}{N^n}\sum_{j \in \frac{1}{N}{\Z}^n\cap[0;1)^n} \psi(j)\phi(u_N(j))-\int_{[0;1]^n}\psi(x)\phi(u_N(x))\mathrm{d}x}\tend{N}{\infty}0.
  \end{equation}
\end{lemma}

\begin{proof}[Proof of \eqref{eq:discretize}]
  First we prove property (C). Given a sequence $(u_N)$ such that for any $N$, $u_N \in E_N$, and $(\mathcal{F}_{\epsilon,T,N}(u_N))$ is bounded, it is easy to show that $(u_N)$ converges weakly to a function $u\in H^1([-T/2;T/2]^n)$, such that $\int u = m$. We are left to show that $u$ takes its values in the interval
  $[0,1]$. But this is a consequence of lemma \ref{lemma:riemann sum trigo} applied to a sequence of functions $(\phi_i)$ that converges from above to the indicator function of $[0,1]$, and $\psi=1$.
  As for the $\Gamma$-convergence, the only problematic terms are $\mathcal{W}_N(u)$ and $\mathcal{U}_{\beta,A,N}(u)$. They can also be taken care of with lemma \ref{lemma:riemann sum trigo}.
\end{proof}

\subsection{Numerical results}\label{subsec:numerical results}
In dimension $n=3$ and for $\alpha=2$, R. Choksi and M. Peletier conjectured the following (see \cite[Conjecture 6.1]{diblock}):
\begin{conjecture}\label{con:only balls1}
  As long as there is a minimizer in \eqref{eq:liquid drop}, it is a ball. Also, when there is no minimizer, the infimum of the energy is attained by a finite number of balls of the same volume, infinitely far away from each other.
\end{conjecture}
In any dimension $n\geq 2$, for $\alpha$ close enough to $n$, this is mostly a theorem of M. Bonacini and R. Cristoferi (see \cite[Theorem 2.12]{bonacini_cristoferi}). Our numerical results suggest that in dimension $2$, the conjecture holds for any $\alpha \in (0,2)$ (\textit{i.e.} the hole admissible range). Note that if the conjecture holds, we can compute explicitely the mass $m_1(n,\alpha)>0$ such that there is a minimizer in \eqref{eq:liquid drop} if and only if $m<m_1$.
Indeed, given $m>0$, let us set
\[
f(m)=P(B[m])+\V{B[m]}.
\]
Then define $m_k$ as the unique solution of
\[kf(\frac{m}{k})=(k+1)f(\frac{m}{k+1}).
\]
Note that $kf(\frac{m}{k})$ is the energy of $k$ balls of volume $m/k$, infinitely far away from each other. Using the homogeneity of $P$ and $\mathcal{V}_\alpha$ we find that
\begin{equation}\label{eq:critical mass}
  m_k = \abs{B}\left(\frac{(k+1)^{\frac1n}-k^{\frac1n}}{(k)^{-\frac{\alpha}{n}}-(k+1)^{-\frac{\alpha}{n}}}\frac{P(B)}{\V{B}}\right)^{\frac{n}{1+\alpha}}.
\end{equation}
We also set $m_0=0$. The sequence $(m_k)$ is increasing. Then an equivalent formulation of conjecture \ref{con:only balls1} is:
\begin{conjecture}\label{con:only balls 2}
  In dimension $n=2$, if $m \in [m_{k-1},m_k]$,
  \[ \inf_{E\subset \rn, \abs{E}=m}  P(E)+\V{E} = kf(\frac{m}{k}).
  \]
  In particular, as long as there is a minimizer in \eqref{eq:liquid drop}, it is a ball. When there is no minimizer, in some sense an optimal set is given by $k$ balls of the same volume infinitely far from each other.
\end{conjecture}

To get minimizers of \eqref{eq:liquid drop} for different volume constraint, we set the volume constraint to $1$ and add a constant $c_m$ to the term $\mathcal{V}_\alpha$. Indeed, minimizing
\[
  \inf_{E\subset \rn, \abs{E}=1}  P(E)+c_m \V{E}
\]
is equivalent to minimizing \eqref{eq:liquid drop} provided
\begin{equation}\label{eq:calpha}
c_m = m^{\frac{1+\alpha}{n}}.
\end{equation} The choice of $T$ is made so that, if $\1_{B[1]}^N$ is the discretization of the ball of volume $1$ with side step $T/N$, we have
\[ \frac{\vat(\1_{B[1]}^N)-\V{B[1]}}{\V{B[1]}}\leq 1\%.\]
Meanwhile, given the number of discretization points $N=2^{11}$, we can't increase $T$ too much, otherwise the discretization of candidate minimizers is less and less precise.

For instance, for $\alpha=1$ and $n=2$, we have
\begin{itemize}
  \item for $T=5\pi$: $\frac{\vat(\1_B^N)-\V{B[1]}}{\V{B[1]}}\simeq 0.08$,
  \item for $T=10\pi$: $\frac{\vat(\1_B^N)-\V{B[1]}}{\V{B[1]}}\simeq 0.04$,
  \item for $T=20\pi$: $\frac{\vat(\1_B^N)-\V{B[1]}}{\V{B[1]}}\simeq 0.01$.
\end{itemize}
These numerical estimates lead us to chose $T=20\pi$. See appendix \ref{ap:computation} for the method used to compute $\V{B[1]}$.

We display the results obtained for $\alpha=1$, and $c_m=1.5,1.6$ in figure \ref{fig:1.5 and 1.6}.
\begin{figure}[h]
\subfloat[$c_m=1.5$]{\includegraphics[scale=0.42]{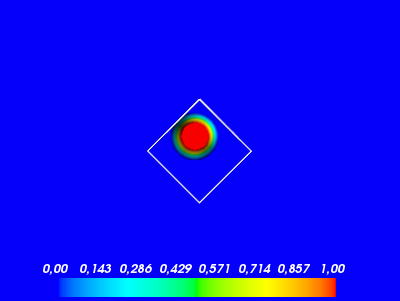}\label{fig:1.5}}
\quad
\subfloat[$c_m=1.6$]{\includegraphics[scale=0.42]{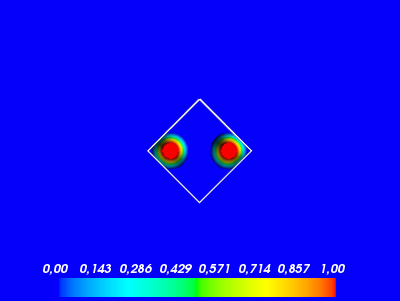}\label{fig:1.6}}
\caption{$\alpha=1$, $A=0$, $\Omega$ is a square.}
\label{fig:1.5 and 1.6}
\end{figure}

Here the box $\Omega$ in which all functions are supported (see subsection \ref{subsec:method}) has been chosen to be a square of diagonal length $\pi$ (and is represented by white lines). We emphasize that this box is needed to get the right minimizers, both theoretically and numerically. Theoretically, the condition that functions are supported in a fixed bounded set is needed for the compactness property (C) (again see subsection \ref{subsec:method}) to be satisfied, both in \emph{step one} and in \emph{step two}, as we let the size of the domain $T$ go to infinity. Numerically, without this box, for $c_m=1.5$, simulations yield two balls (instead of one as shown on figure \ref{fig:1.5}) that get further and further away from each other as $T$ increases. But this configuration \emph{does not} converge to an admissible candidate, so it definitely doesn't converge to a minimizer.

We observe that for $c_m=1.6$, we get two balls in opposite corners of the square $\Omega$: it is consistent with the expected repulsive behaviour of the non-local term $\mathcal{V}_\alpha$.
Moreover, using \eqref{eq:critical mass} and \eqref{eq:calpha}, we find that, if conjecture \ref{con:only balls 2} is true, there must be a minimizer up to $c_m \simeq 1.67$. Numerically, we find that there is a minimizer up to a constant $c_m \in(1.5,1.6)$, which is relatively close to $1.67$.
We also observe similar results for different values of $\alpha$, including in the near field-dominated regime $\alpha <1$.

For $\Omega$ a disk of diameter $\pi$, if one increases further $c_m$, we get three balls located near the boundary of $\Omega$, as shown in figure \ref{fig:3.0} for $c_m=3.0$. This is consistent with the conjecture that the energy is minimized by balls of the same volume. To illustrate the effect of the confining potential, we display in figure \ref{fig:confining_potential} the minimizer for $c_m=3.0$, $A=1$ and $\beta=16$.

\begin{figure}[h]
\subfloat[$A=0$]{\includegraphics[scale=0.42]{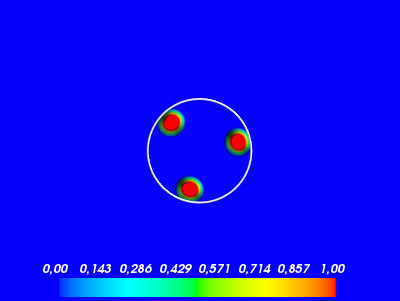}\label{fig:3.0}}
\quad
\subfloat[$A=1$]{\includegraphics[scale=0.42]{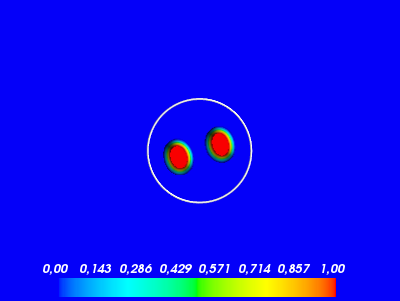}\label{fig:confining_potential}}
\caption{$\alpha=1$, $c_m=3.0$, $\beta=16$, $\Omega$ is a disk.}
\label{fig:1.5 and 1.6}
\end{figure}

Finally, let us mention that the number of discretization points is $N=2^{11}$ in each direction. Numerical minimization is made using the solver IPOPT \cite{ipopt}. The computation time on a standard computer is about an hour.

\appendix
\section{Appendix: Study of the sets $\widetilde{S}_*$ and $S_*$ from the proof of proposition \ref{prop:stability}}\label{ap:variation}
In this appendix, we give explicit forms of the sets $\widetilde{S}_*$ and $S_*$, needed in the proof of proposition \ref{prop:stability}.
Let us define, for all $k \geq 2$, a function $f_k : (0,\infty) \to \R$ by
\[f_k(\gamma)=1-\gamma^{1+\alpha}\frac{\mu_k^\alpha-\mu_1^\alpha}{\lambda_k-\lambda_1} +\gamma^{1+\beta} \frac{A \beta}{\lambda_k-\lambda_1}.\]
Then the sets  $\widetilde{S}_*$ and $S_*$ from the proof of proposition \ref{prop:stability} are defined by
\[\widetilde{S}_*=\{\forall k\geq 2,f_k >0\} \quad \text{and} \quad S_*=\{\forall k \geq 2,f_k \geq 0\}.\]
As stated in \cite[equations (7.4), (7.5) and (7.6)]{vincent}, we have
\begin{align}\label{eq:mukalpha}
\mu_k^{\alpha}=\begin{cases}
			\frac{2^{1+\alpha}\pi^{ \frac{n-1}{2}}}{1-\alpha}\frac{\Gamma(\frac{1+\alpha}{2})}{\Gamma(\frac{n-\alpha}{2})}\left(\frac{\Gamma(k+\frac{n-\alpha}{2})}{\Gamma(k+\frac{n-2+\alpha}{2})}- \frac{\Gamma(\frac{n-\alpha}{2})}{\Gamma(\frac{n-2+\alpha}{2})}\right) &\text{if }\alpha \in (0,1), \\
			2^{\alpha}\pi^{ \frac{n-1}{2}}\frac{\Gamma(\frac{\alpha-1}{2})}{\Gamma(\frac{n-\alpha}{2})}\left( \frac{\Gamma(\frac{n-\alpha}{2})}{\Gamma(\frac{n-2+\alpha}{2})}-\frac{\Gamma(k+\frac{n-\alpha}{2})}{\Gamma(k+\frac{n-2+\alpha}{2})}\right) &\text{if } \alpha \in (1,n), \\
			\frac{4\pi^{ \frac{n-1}{2}}}{\Gamma(\frac{n-1}{2})}\left(\frac{\Gamma'(k+\frac{n-1}{2})}{\Gamma(k+\frac{n-1}{2})}- \frac{\Gamma'(\frac{n-1}{2})}{\Gamma(\frac{n-1}{2})}\right) &\text{if }\alpha =1. \\
		      \end{cases}
\end{align}
Recall also that for any $k \geq 0$, $\lambda_k=k(n+k-2)$. Now let us treat each case enumerated in proposition \ref{prop:stability} separately.

\emph{Case \ref{item:alpha>beta}: $\alpha>\beta$.} A simple study of the sign of $f_k'$ shows that each $f_k$ is increasing from $0$ to a point $\gamma_k$, then decreasing from $\gamma_k$ to $+\infty$. What is more $f_k(0)=1$ and $\lim_{+\infty}f_k = -\infty$, so each $f_k$ has exactly one zero and is positive left of this zero and negative right of it. At last, for any constant $K>0$, $f_k(\gamma)\tend{k}{\infty}1$ uniformly in $\gamma\leq K$. Putting these facts together shows that the sets $\widetilde{S}_*$ and $S_*$ have the forms:
\[\widetilde{S}_*=(0,m_*) \quad  \text{and} \quad  S_*=(0,m_*],\]
for some critical mass $m_*>0$ (depending on $n$, $\alpha$, $\beta$ and $A$).

\emph{Case \ref{item:alpha=beta}: $\alpha=\beta$.} For any $k$, $f_k$ is either decreasing or increasing or constant (depending on the size of $A$). If none of them is decreasing, then for any $k$, $f_k\geq f_k(0)=1$, so
\[\widetilde{S}_*=S_*=(0,+\infty).\]
Otherwise the same arguments as in \emph{case one} shows again that
\[\widetilde{S}_*=(0,m_*) \quad  \text{and} \quad  S_*=(0,m_*],\]
for some critical mass $m_*>0$ (depending on $n$, $\alpha$, $\beta$ and $A$).

\emph{Cases \ref{item:alpha<beta>1}, \ref{item:alpha<beta<1} and \ref{item:alpha<beta=1}: $\alpha<\beta$.} A simple study of the sign of $f_k'$ shows that each $f_k$ is decreasing from $0$ to a point $\gamma_k$, then increasing from $\gamma_k$ to $+\infty$, and we have:
\begin{equation}\label{eq:argminfk}
\gamma_k=\left( \frac{1+\alpha}{A \beta(1+\beta)}(\mu_k^\alpha-\mu_1^\alpha)\right)^{ \frac{1}{\beta-\alpha}}.
\end{equation}
Another simple computation shows that
\begin{equation}\label{eq:minfk}
\min f_k = f_k(\gamma_k) = 1-\left( \frac{1+\alpha}{A \beta(1+\beta)}\right)^{ \frac{1+\alpha}{\beta-\alpha}}\frac{\beta-\alpha}{1+\beta}\frac{(\mu_k^\alpha-\mu_1^\alpha)^{ \frac{1+\beta}{\beta-\alpha}}}{\lambda_k-\lambda_1}.
\end{equation}
We must treat the subcases $\alpha>1$, $\alpha=1$ and $\alpha<1$ separately.
\newline\emph{Subcase one: $\alpha>1$}. We use the following classical Stirling formula:
\[\Gamma(x) \underset{x\to +\infty}{\sim} \sqrt{ {\frac{2\pi}{x}}}\left( \frac{x}{e}\right)^x,\]
to find that
\[ \frac{\Gamma(k+\frac{n-\alpha}{2})}{\Gamma(k+\frac{n-2+\alpha}{2})} \underset{k\to \infty}{\sim} k^{1-\alpha}.\]
With \eqref{eq:mukalpha}, this means that the sequence $(\mu_k^\alpha)$ is bounded. As $\lambda_k \tend{k}{\infty}\infty$, we get from \eqref{eq:minfk}
\[\min f_k \tend{k}{\infty}1.\]
Thus there exists an index $k_0$ such that
\begin{equation}
\widetilde{S}_*=\bigcap_{k=2}^{k_0}\{f_k>0\} \quad \text{and} \quad S_*=\bigcap_{k=2}^{k_0}\{f_k\geq 0\}.
\end{equation}
As for any $k$, $\lim_{+\infty} f_k = +\infty$, we get that $\widetilde{S}_*$ and $S_*$ both contain an unbounded interval, which is what we wanted.
\newline\emph{Subcase two: $\alpha=1$.} We use the classical asymptotics of the digamma function $\frac {\Gamma'}{\Gamma}$:
\[\frac {\Gamma'}{\Gamma}(x) \underset{x\to +\infty}{\sim} \ln(x), \]
to find that, according to \eqref{eq:minfk},
\[ \min f_k \tend{k}{+\infty} 1.\]
We conclude as above.
\newline\emph{Subcase three: $\alpha<1$.} Once again we use the asymptotics
\[ \frac{\Gamma(k+\frac{n-\alpha}{2})}{\Gamma(k+\frac{n-2+\alpha}{2})} \underset{k\to \infty}{\sim} k^{1-\alpha},\]
to find that
\begin{equation}\label{eq:asymptotics}
  \frac{(\mu_k^\alpha-\mu_1^\alpha)^{\frac{1+\beta}{\beta-\alpha}}} {\lambda_k-\lambda_1} \underset{k\to \infty}{\sim} \frac{(k^{1-\alpha})^{\frac{1+\beta}{\beta-\alpha}}}{k^2} = k^{\frac{(1-\beta)(1+\alpha)}{\beta-\alpha}}.
\end{equation}
If $\beta>1$, once again we have
\[ \min f_k \tend{k}{+\infty} 1,\]
and we conclude as above. If $\beta<1$, then we have
\begin{equation}\label{eq:minfkinfini}
  \min f_k \tend{k}{+\infty} -\infty
\end{equation}
Also, by definition we have
\[ f_{k+1}(\gamma_k)-1 = -\gamma_k^{1+\alpha}\frac{\mu_{k+1}^\alpha-\mu_1^\alpha}{\lambda_{k+1}-\lambda_1} +\gamma_k^{1+\beta} \frac{A \beta}{\lambda_{k+1}-\lambda_1}.\]
As the sequence $(\mu_k^\alpha)$ is increasing we get
\begin{align}
  f_{k+1}(\gamma_k)-1 &\leq -\gamma_k^{1+\alpha}\frac{\mu_{k}^\alpha-\mu_1^\alpha}{\lambda_{k+1}-\lambda_1} +\gamma_k^{1+\beta} \frac{A \beta}{\lambda_{k+1}-\lambda_1} \nonumber\\
                      & = (f_k(\gamma_k)-1)\frac{\lambda_k-\lambda_1}{\lambda_{k+1}-\lambda_1} \nonumber\\
                      & \tend{k}{\infty}-\infty. \label{eq:fkgk}
\end{align}
With \eqref{eq:minfkinfini}, this means that
\[ [\gamma_k,\gamma_{k+1}]\subset (\widetilde{S}_*)^\mathsf{c}\cap (S_*)^\mathsf{c}.\]
What is more, from \eqref{eq:argminfk}, we have $\gamma_k \tend{k}{+\infty} +\infty$, so $\widetilde{S}_*$ and $S_*$ are both bounded, which is what we wanted. At last, if $\beta=1$ we find using $\eqref{eq:asymptotics}$ that there exists a constant $C_\alpha$ such that
\[ \min f_k \tend{k}{\infty}1-\frac{C_\alpha}{A^{\frac{1+\alpha}{\beta-\alpha}}}.\]
For $A^{\frac{1+\alpha}{\beta-\alpha}} > C_\alpha$, we conclude as in \emph{case one}. For $A^{\frac{1+\alpha}{\beta-\alpha}} < C_\alpha$, we conclude as above that $\widetilde{S}_*$ and $S_*$ are both bounded. For $A^{\frac{1+\alpha}{\beta-\alpha}} < C_\alpha$, we have to use the following more precise form of Stirling's approximation: \[\Gamma(x) \underset{x\to +\infty}{\sim} \sqrt{ {\frac{2\pi}{x}}}\left( \frac{x}{e}\right)^x(1+O(\frac 1x)).\]
Proceeding to simple asymptotic expansions, we find that for $k$ large enough we have
\[ \min f_k >0.\]
We conclude as in \emph{case one}.

\section{Computation of $\V{B[1]}$}\label{ap:computation}
Here we explain how we compute $\V{B[1]}$ numerically, as needed in subsection \ref{subsec:numerical results} to choose the value of $T$. In order to compute numerically the improper integral
\[\V{B[1]} = \int_{B[1]\times B[1]} \frac{\mathrm{d}x\mathrm{d}y}{\abs{x-y}^{2-\alpha}}, \]
we add a small term $\epsilon>0$ to the denominator of the integrand. So we compute
\[\mathcal{V}_{\alpha,\epsilon}(B[1])=\int_{B[1]\times B[1]}\frac{\mathrm{d}x\mathrm{d}y}{\abs{x-y}^{2-\alpha}+\epsilon}.\]
To control the error introduced by the parameter $\epsilon$, we need to estimate the difference $\Delta_\epsilon := \V{B[1]} - \mathcal{V}_{\alpha,\epsilon}(B[1])$. We have
\begin{align*}
  \Delta_\epsilon & =\int_{B[1]\times B[1]} \frac{\mathrm{d}x\mathrm{d}y}{\abs{x-y}^{2-\alpha}}-\int_{B[1]\times B[1]}\frac{\mathrm{d}x\mathrm{d}y}{\abs{x-y}^{2-\alpha}+\epsilon} \\
  & =\int_{B[1]\times B[1]}\frac{\epsilon \mathrm{d}x\mathrm{d}y} {\abs{x-y}^{2-\alpha}(\abs{x-y}^{2-\alpha}+\epsilon)} \\
  & \leq \int_{B[1]\times B[1]}\1_{\abs{x-y}<r}\frac{\mathrm{d}x\mathrm{d}y}{\abs{x-y}^{2-\alpha}}+\int_{B[1]\times B[1]}\1_{\abs{x-y}\geq r}\frac{\epsilon\mathrm{d}x\mathrm{d}y}{\abs{x-y}^{2-\alpha}r^{2-\alpha}} \\
  & \leq \int_{B[1]\times \R^2}\1_{\abs{x-y}<r}\frac{\mathrm{d}x\mathrm{d}y}{\abs{x-y}^{2-\alpha}} + \frac{\epsilon}{r^{2-\alpha}}\int_{B[1]\times B[1]}\frac{\mathrm{d}x\mathrm{d}y}{\abs{x-y}^{2-\alpha}} \\
  & \leq \int_{B[1]\times \R^2}\1_{\abs{y}<r}\frac{\mathrm{d}x\mathrm{d}y}{\abs{y}^{2-\alpha}} + \frac{\epsilon}{r^{2-\alpha}} \int_{B[1]}\int_{B[1]}\frac{\mathrm{d}x\mathrm{d}y}{\abs{y}^{2-\alpha}} \\
  & = \int_{\abs{y}<r}\frac{\mathrm{d}y}{\abs{y}^{2-\alpha}} + \frac{\epsilon}{r^{2-\alpha}}\int_{B[1]}\frac{\mathrm{d}y}{\abs{y}^{2-\alpha}} \\
  & = 2\pi\int_0^r \frac{r^{2-1}\mathrm{d}r}{r^{2-\alpha}}+\frac{\epsilon}{r^{2-\alpha}}2\pi\int_0^{\frac{1}{\sqrt{\pi}}} \frac{r^{2-1}\mathrm{d}r}{r^{2-\alpha}} \\
  & = \frac{2\pi}{\alpha}(r^\alpha+\frac{\epsilon }{r^{2-\alpha}\pi^{\frac{\alpha}{2}}}),
\end{align*}
for some $r>0$. This last bound attains its minimum for $r=\left(\frac{(2-\alpha)\epsilon}{\alpha \pi^{\frac{\alpha}{2}}}\right)^{\frac{1}{2}}$. From there we deduce
\begin{equation*}
  \Delta_\epsilon \leq \frac{2\pi}{\alpha}\frac{2}{2-\alpha}\left(\frac{2-\alpha}{\alpha}\right)^{\frac{\alpha}{2}}\left(\frac{\epsilon}{\pi^{\frac{\alpha}{2}}}\right)^{\frac{\alpha}{2}}.
\end{equation*}
With $\alpha=1$, we get
\begin{equation*}
  \Delta_\epsilon \leq 4 \pi^{\frac{3}{4}}\sqrt{\epsilon}.
\end{equation*}
Now the proper integral $\mathcal{V}_{\alpha,\epsilon}(B[1])$ can be expressed in polar coordinates, and computed with arbitrary precision in the Julia language, using the \emph{HCubature} package.
\bibliographystyle{plain}
\bibliography{biblio}

\bigskip\noindent
François Générau:
Laboratoire Jean Kuntzmann (LJK),
Universit\'e Joseph Fourier\\
Bâtiment IMAG, 700 avenue centrale,
38041 Grenoble Cedex 9 - FRANCE\\
{\tt francois.generau@univ-grenoble-alpes.fr}\\

\bigskip\noindent
Edouard Oudet:
Laboratoire Jean Kuntzmann (LJK),
Universit\'e Joseph Fourier\\
Bâtiment IMAG, 700 avenue centrale,
38041 Grenoble Cedex 9 - FRANCE\\
{\tt edouard.oudet@univ-grenoble-alpes.fr}\\
{\tt http://www-ljk.imag.fr/membres/Edouard.Oudet/}

\end{document}